\documentclass[preprint,10pt]{elsarticle}

\usepackage{amssymb}
\usepackage{amsmath}
\usepackage[all]{xy}
\usepackage{latexsym}
\usepackage{mathtools}
\usepackage{amsthm}
\usepackage{a4wide}
\usepackage{color}
\usepackage{hyperref}
\input diagxy
\journal{}

\theoremstyle{plain}
  \newtheorem{thm}{Theorem}[section]
  \newtheorem{lem}[thm]{Lemma}
  \newtheorem{prop}[thm]{Proposition}
  \newtheorem{cor}[thm]{Corollary}
\theoremstyle{definition}
  \newtheorem{defn}[thm]{Definition}
  
  \newtheorem{exmp}[thm]{Example}
  \newtheorem{rem}[thm]{Remark}

\DeclareMathOperator{\dom}{dom}
\DeclareMathOperator{\cod}{cod}

\DeclareMathOperator{\ob}{ob}

\makeatletter
\def\oarrowfill@#1#2#3#4#5{%
  $\m@th\thickmuskip0mu\medmuskip\thickmuskip\thinmuskip\thickmuskip
   \relax#5#1\mkern-7mu%
   \cleaders\hbox{$#5\mkern-2mu#2\mkern-2mu$}\hfill
   \mathclap{#3}\mathclap{#2}%
   \cleaders\hbox{$#5\mkern-2mu#2\mkern-2mu$}\hfill
   \mkern-7mu#4$%
}
\def\orightarrowfill@{%
  \oarrowfill@\relbar\relbar\circ{\to<100>}}
\newcommand\xorightarrow[2][]{%
  \ext@arrow 0055{\orightarrowfill@}{#1}{#2}}
\makeatother

\makeatletter
\def\ps@pprintTitle{%
 \let\@oddhead\@empty
 \let\@evenhead\@empty
 \def\@oddfoot{\centerline{\thepage}}%
 \let\@evenfoot\@oddfoot}
\makeatother

\newcommand{\da}{\downarrow}
\newcommand{\Da}{\Downarrow}
\newcommand{\ua}{\uparrow}
\newcommand{\ra}{\rightarrow}

\newcommand{\lda}{\swarrow}
\newcommand{\rda}{\searrow}

\newcommand{\Lra}{\Longrightarrow}

\newcommand{\rat}{\!\rightarrowtail\!}
\newcommand{\xora}{\xorightarrow}
\newcommand{\oto}{\xora{\ \ \ \ }}

\newcommand{\bv}{\bigvee}
\newcommand{\bw}{\bigwedge}

\newcommand{\dv}{\dashv}

\newcommand{\nat}{\natural}

\newcommand{\si}{\sigma}

\newcommand{\CA}{\mathcal{A}}
\newcommand{\CB}{\mathcal{B}}
\newcommand{\CC}{\mathcal{C}}
\newcommand{\CD}{\mathcal{D}}
\newcommand{\CE}{\mathcal{E}}

\newcommand{\CJ}{\mathcal{J}}

\newcommand{\CP}{\mathcal{P}}
\newcommand{\CQ}{\mathcal{Q}}

\newcommand{\CT}{\mathcal{T}}

\newcommand{\CX}{\mathcal{X}}

\newcommand{\Bf}{{\bf f}}
\newcommand{\Bg}{{\bf g}}
\newcommand{\Bh}{{\bf h}}

\newcommand{\sI}{{\sf I}}

\newcommand{\sP}{{\sf P}}
\newcommand{\sPd}{{\sf P}^{\dag}}

\newcommand{\sS}{{\sf S}}

\newcommand{\sY}{{\sf Y}}

\newcommand{\Cat}{{\bf Cat}}
\newcommand{\CAT}{{\bf CAT}}
\newcommand{\CATB}{{\CAT\Da_c\CB}}
\newcommand{\CatB}{{\Cat\Da_c\CB}}

\newcommand{\Chu}{{\bf Chu}}
\newcommand{\CHU}{{\bf CHU}}

\newcommand{\DIS}{{\bf DIS}}

\newcommand{\Met}{{\bf Met}}
\newcommand{\Ord}{{\bf Ord}}

\newcommand{\QUANT}{{\bf QUANT}}
\newcommand{\Rel}{{\bf Rel}}

\newcommand{\Set}{{\bf Set}}

\newcommand{\SUP}{{\bf SUP}}

\newcommand{\QCAT}{\CQ\text{-}\CAT}
\newcommand{\QCHU}{\CQ\text{-}\CHU}
\newcommand{\QCat}{\CQ\text{-}\Cat}
\newcommand{\QChu}{\CQ\text{-}\Chu}

\newcommand{\co}{{\rm co}}

\newcommand{\op}{{\rm op}}

\newcommand{\dPhi}{\Phi^{\da}}
\newcommand{\uPhi}{\Phi_{\ua}}

\newcommand{\dPsi}{\Psi^{\da}}
\newcommand{\uPsi}{\Psi_{\ua}}
\newcommand{\dXi}{\Xi^{\da}}
\newcommand{\uXi}{\Xi_{\ua}}

\newcommand{\PB}{\CP\CB}
\newcommand{\PC}{\sP\CC}
\newcommand{\PD}{\sP\CD}
\newcommand{\PE}{\sP\CE}
\newcommand{\PJ}{\sP\CJ}

\newcommand{\PdB}{\CP^{\dag}\CB}

\newcommand{\PdD}{\sP^{\dag}\CD}
\newcommand{\PdE}{\sP^{\dag}\CE}
\newcommand{\PdJ}{\sP^{\dag}\CJ}
\newcommand{\sYd}{\sY^{\dag}}

\newcommand{\QB}{\CQ_{\CB}}

\begin{document}

\begin{frontmatter}

%% Title, authors and addresses

%% use the tnoteref command within \title for footnotes;
%% use the tnotetext command for theassociated footnote;
%% use the fnref command within \author or \address for footnotes;
%% use the fntext command for theassociated footnote;
%% use the corref command within \author for corresponding author footnotes;
%% use the cortext command for theassociated footnote;
%% use the ead command for the email address,
%% and the form \ead[url] for the home page:
%% \title{Title\tnoteref{label1}}
%% \tnotetext[label1]{}
%% \author{Name\corref{cor1}\fnref{label2}}
%% \ead{email address}
%% \ead[url]{home page}
%% \fntext[label2]{}
%% \cortext[cor1]{}
%% \address{Address\fnref{label3}}
%% \fntext[label3]{}

\title{The Fundamental Group as the Structure of a Dually Affine Space}

%% use optional labels to link authors explicitly to addresses:
%% \author[label1,label2]{}
%% \address[label1]{}
%% \address[label2]{}

\author{Eraldo Giuli\fnref{A}}
\ead{eraldo.giuli@gmail.com}

\author{Walter Tholen\corref{cor}\fnref{A}}
\ead{tholen@mathstat.yorku.ca}

\address{Department of Mathematics and Statistics, York University, Toronto, Ontario, Canada, M3J 1P3}
\address{\rm Dedicated to the memory of Horst Herrlich}

\cortext[cor]{Corresponding author.}
\fntext[A]{Partial financial assistance by Botswana International University of Science and Technology (BIUST) and by the Natural Sciences and Engineering Research Council (NSERC) of Canada is gratefully acknowledged. This work was completed while the first author held a Visiting Professorship at BIUST in Palapye, Botswana.}

\begin{abstract}
This paper dualizes the setting of affine spaces as originally introduced by Diers for application to algebraic geometry and expanded upon by various authors, to show that the fundamental groups of pointed topological spaces appear as the structures of dually affine spaces. The dual of the Zariski closure operator is introduced, and the 1-sphere and its copowers together with their fundamental groups are shown to be examples of complete objects with respect to the Zariski dual closure operator.
\end{abstract}

\begin{keyword}
%% keywords here, in the form: keyword \sep keyword
Dually affine space \sep pointed topological space \sep loop space \sep fundamental group \sep topological category \sep Zariski dual closure operator \sep separated dually affine space \sep complete dually affine space 

%% PACS codes here, in the form: \PACS code \sep code

%% MSC codes here, in the form: \MSC code \sep code
%% or \MSC[2008] code \sep code (2000 is the default)
\MSC[2010] 	 18D99 \sep 54A99 \sep 18C05

\end{keyword}

\end{frontmatter}

%% \linenumbers

%% main text
\section{Introduction}

With the algebraic theory of commutative $K$-algebras (for a field $K$) serving as his role model, in \cite{Diers} Diers presented a simple categorical setting that allowed him to define and study {\em affine sets modelled by} $\CT$ (and $K$) --$\CT$-{\em sets} for short--, for any (finitary or infinitary) algebraic theory $\CT$ pertaining to a Birkhoff variety of general algebras. The setting provided an efficient framework for deriving a long list of concrete dualities (in the sense of \cite{Dimov Tholen, Porst Tholen}) that subsequently has been further extended by other authors; see in particular \cite{Giuli Hofmann}. In Diers' role model, for the algebraic set $X=K^n, n$ any cardinal, that in his setting comes equipped with the $K$-algebra of $K$-valued polynomial functions on $X$, one is especially interested in those subsets of $X$
that are the zero sets of some set of polynomials in $K[x_i]_{i\in n}$, {\em i.e.}, in the Zariski-closed subsets of $X$ that then get equipped with the restrictions of the polynomial functions.

In \cite{GiuliTA}, this paper's first author formulated Diers' setting for an arbitrary category $\CX$ (rather than ${\bf Set}$) and a distinguished $\CX$-object $K$, whose $\CT$-algebraic operations now ``live" in $\CX$, formalizing the Zariski closure as a categorical closure operator in the sense of \cite{clop} and relating
Zariski closed sets to his earlier work on completions with Br\"ummer, Colebunders, Herrlich and others; see, for example, \cite{BGH, DGLC, GiuliT0}. An $\CX$-object modelled by $\CT$ and $K$ has as its structure a $\CT$-subalgebra of $\CX (X,K)$, where the hom-set $\CX(X,K)$ inherits its $\CT$-structure from $K$. 
With the notion of closure operator categorically dualized as in \cite{dual clop}, it is clear that the setting and theory of
\cite{GiuliTA} allow for rather routine formal dualization. The purpose of the present paper is to give a first indication that such undertaking may be quite rewarding in terms of prominent examples and future applications.

While in Diers' setting one considers $\CT$-algebras of ``$K$-valued functionals" in $\CX$, in the dual setting we have re-named $K$ to $S$ in reference to our primary example and consider $\CT$-algebras of ``$S$-paths" in $\CX$. Of course,
for $\CX(X,S)$ to carry a $\CT$-algebra structure, $S$ must be a $\CT$-{\em co}algebra in $\CX$, {\em i.e.}, must come equipped with certain cooperations satisfying laws as dually prescribed by $\CT$. In our principal example, $S={\sf S}^1$ is the 1-sphere considered as an object of the homotopy category of pointed topological spaces. With $\CT$ the theory of groups, it naturally provides the loop space of a pointed topological space with its fundamental group as its dually affine structure. This example requires us to be extremely careful about the use of limits  and colimits. But coproducts do exist in the homotopy category, in particular the copowers of ${\sf S}^1$, and these suffice for the dualization exercise.

Since not all readers may find the dualization of Diers' setting straightforward, we have written this paper in a way which does not require prior reading of \cite{Diers} or \cite{GiuliTA}. Hence, we compactly present the essential properties of the
category of {\em dually affine spaces in} $\CX$ {\em modelled by} $\CT$ {\em and} $S$, as a category over both
$\CX$ and the category of $\CT$-algebras, paying special attention to the existence of dually affine spaces freely generated by a $\CT$-algebra. Other than the principal example we also consider easy examples from algebra; further examples from topology will be included in future work. We then consider the {\em Zariski dual closure operator} for dually affine spaces and the notion of {\em Zariski completeness}. Again, it is important to realize that, while the general treatment of the Zariski dual closure operator requires the existence of colimits --that will generally fail to exist in the homotopy category (see, for example, 
\cite{Arkowitz, Strom})--, it is possible to consider Zariski-closedness and -completeness in the presence of just copowers 
of the distinguished object $S$. 

As a general reference for category theory we cite \cite{AHS, Mac Lane}, and for topology and homotopy theory
standard texts like \cite{Dugundij, Willard} provide sufficient background for this paper. 

\section{Dually affine spaces modelled by a coalgebra}

Let $\CX$ be a category with small hom-sets and $S$ be a distinguished object in $\CX$ which comes with a family
of cooperations (of potentially infinite arities) on $S$ in $\CX$ which may be required to satisfy some equational laws. Here, by a {\em cooperation} $\omega$ on $S$ in $\CX$ of arity $n_{\omega}$ ($n_{\omega}$ a cardinal number) we mean an $\CX$-morphism $\omega : S\to n_{\omega} \cdot S = \coprod_{i<n_{\omega}}S_i $ with $S_i=S$, assuming that the needed copowers of $S$ exist in $\CX$. Of course, the family $(n_{\omega})_{\omega}$ defines a {\em type} (or {\em signature}) $\CT$ as used in universal algebra, and $S$ is simply a $\CT$-algebra in $\CX^{{\rm op}}$. Equational laws are best capture when one describes $\CT$ as an {\em algebraic theory} in the sense of Lawvere (finite arities) or Linton (infinite arities) -- always assuming, however, the existence of free $\CT$-algebras over $\bf Set$, which is certainly guaranteed when the arities are bounded by a fixed cardinal, in particular when they are all finite (see \cite{Manes}).

For every $X$ in $\CX$, the $\CT$-coalgebra structure of $S$ in $\CX$ provides the hom-set 
$\CX(S,X)= {\rm hom}_{\CX}(S,X)$ with a $\CT$-algebra structure in ${\bf Set}$: every cooperation $\omega$ gives the operation

$$\omega_X: \CX(S,X)^{n_{\omega}}\to \CX(S,X), (a_i)_{i<n_{\omega}}\mapsto (S\to^{\omega}n_{\omega}\cdot S\to^{[a_i]_{i<n_{\omega}}}X),$$
where $[a_i]_{i<n_{\omega}}$ is the morphism that equals $a_i$ when restricted to the $i$-th summand of the coproduct $n_{\omega}\cdot S$. Since for every $\CX$-morphism $f:X\to Y$ the map $\CX(S,f):\CX(S,X)\to \CX(S,Y)$ becomes a $\CT$-homomorphism, the covariant hom-functor of $\CX$ represented by $S$ takes values in the category of $\CT$-algebras; so we write

$$\CX(S,-):\CX \to {\rm Alg}(\CT).$$
We are now ready to set up the category

$${\rm Aff}^*_S(\CT,\CX)$$
of {\em dually} {\em affine spaces in} $\CX$ {\em modelled by} $S$ (and $\CT$): its objects are $\CX$-objects $X$ that come with a $\CT$-subalgebra $A$ of $\CX(S,X)$ (we write $A\leq\CX(S,X)$); its morphisms $f:(X,A)\to (Y,B)$ are
$\CX$-morphisms $f:X\to Y$ such that $\CX(S,f)$ maps $A$ into $B$, that is $\{f\}\cdot A \subseteq B$. Besides the forgetful functor

$$U:{\rm Aff}^*_S(\CT,\CX) \to \CX, (X,A)\mapsto X,$$
one also has the {\em structure functor}

$$\Gamma: {\rm Aff}^*_S(\CT,\CX)\to {\rm Alg}(\CT), (X,A)\mapsto A,$$
both of which will be of interest later on.

\begin{rem}\label{Diers}
For $\CX={\bf Set}^{\rm op}$, $\CT$ a Lawvere-Linton theory and $S$ a fixed $\CT$-algebra,  
${\rm Aff}^*_S(\CT,\CX)$ 
is the dual of the category of $\CT$-{\em sets}, as first introduced by Diers in \cite{Diers} and called {\em affine sets over} $S$ in \cite{DiersGeom}. Replacing sets by an arbitrary category $\CX$, the paper \cite{GiuliTA} and others
extended Diers' setting and studied the dual of the category ${\rm Aff}^*_S(\CT,\CX^{\rm op})$, calling its objects 
{\em affine} $\CX$-{\em objects modelled by} $S$, where the $\CX$-object $S$
now comes with a family of operations $\omega:X^{n_{\omega}}\to X$ in $\CX$, {\em i.e.}, $S$ is a $\CT$-algebra in $\CX$. Since in the papers cited above and in \cite{Giuli Hofmann} one already finds an extensive list of examples and applications of the Diers setting, in what follows we restrict ourselves to discussing those new examples that motivated this paper's study of 
${\rm Aff}^*_S(\CT,\CX)$.
\end{rem}

\begin{exmp}\label{examples first}
(1) For $\CX= {\bf Set}, S=1$ a singleton set and $\CT=\emptyset$ the empty type, ${\rm Aff}^*_1(\CT,\CX)$ is the quasitopos {\rm Sub}({\bf Set}) of sets $X$ equipped with a subset $A$, with maps preserving the distinguished subsets. The same category is obtained if $\CT$ is the type with one unary operation, since the only unary cooperation on 1 provides every set $X \cong {\bf Set}(1,X)$ with the identical unary operation.

(2) For a unital ring $R$ let $\CX={\bf Mod}_R$ be the category of (left-) $R$-modules. With $\CT=\emptyset$ again, 
the objects of ${\rm Aff}^*_R(\CT,\CX)$ may be described as $R$-modules $X$ equipped with a subset $A\subseteq X$, since $X\cong {\rm hom}_R(R,X)$; morphisms are $R$-linear maps preserving the distinguished subsets.

(3) Choosing $\CX={\bf Mod}_R$ again, let us now consider the $R$-linear cooperations

$$\delta: R\to R\oplus R, 1\mapsto (1,1), \; {\rm and} \;  \alpha(-): R\to R, 1\mapsto \alpha,$$
for every $\alpha \in R$. They reproduce the $R$-module operations

$$(-)\!+\!(-): X\times X \to X \; {\rm and} \; \alpha(-):X\to X$$
on every $R$-module $X$. Hence, with $\CT$ denoting the type consisting of one binary and 
$R$-many unary operations or, equivalently, the algebraic theory of $R$-modules, ${\rm Aff}^*_R(\CT,\CX)$ is the category ${\rm Sub}({\bf Mod}_R)$ of $R$-modules $X$ equipped with a submodule $A\leq X$; morphisms are $R$-linear maps preserving the distinguished submodules.

(4) Consider the category $\CX ={\bf Top}_{\bullet}$ of pointed topological spaces with continuous maps that preserve the distinguished ``base" points, and $S={\sf {S}}^1=[0,1]/(0\!\sim\!1)$ the 1-sphere with distinguished point $0=1$, provided with the binary and unary cooperations

$$\gamma:{\sf S}^1\to {\sf S}^1\vee {\sf S}^1, \; t\mapsto 
\left\{
\begin{array}{lr}
(2t,0)&{\rm if}\;t\le \frac{1}{2}\\
(2t-1,1)&{\rm if}\; t\ge \frac{1}{2}
\end{array}
\right\},
{\rm and}\; \tau:{\sf S}^1\to {\sf S}^1,\;t\mapsto1-t,
$$
where ${\sf S}^1\vee {\sf S}^1=({\sf S}^1\!\times \!\{0\})+({\sf S}^1\!\times\! \{1\})/((0,0)\!\sim\!(1,1))$ denotes the coproduct in ${\bf {Top}}_{\bullet}$, as well as with the trivial nullary cooperation ${\sf S}^1\to \bullet$, where $\bullet$ is the zero object of ${\bf Top}_{\bullet}$. For $\CT$ the type with one binary and one self-inverse unary operation, ${\rm Aff}^*_{{\bf S}^1}(\CT,{\bf Top}_{\bullet})$
has as objects pointed topological spaces $X=(X,x_0)$ that come equipped with a set $A$ of loops in $X$ (with $x_0$ as their common start- and endpoint), such that $A$ contains the constant loop in $X$ and is closed under the {\em concatenation} and {\em twist} operations

$$\gamma_X: \Omega X\!\times\!\Omega X \to \Omega X \;{\rm and}\; \tau_X:\Omega X\to \Omega X,$$
where $\Omega X:={\bf {Top}}_{\bullet}({\sf S}^1,X);$
morphisms are continuous maps that preserve the base points and the attached $\CT$-algebras of loops.
Of course, when provided with the compact-open topology, the set $\Omega X$ becomes the {\em loop space} of $X$
which, with its operations, is a so-called ${\rm A}_{\infty}$-space.

(5) Let $\CX={\bf hTop}_{\bullet}={\bf Top}_{\bullet}/\!\simeq$ be the homotopy category of pointed topological spaces, with $\simeq$ denoting the pointed homotopy relation: for $f_0, f_1: (X, x_0)\to (Y,y_0)$ in ${\bf Top}_{\bullet}$ one writes $f_0\simeq f_1$ if some continuous map 
$\varphi: X\times I \to Y$ (with $I=[0,1]$) satisfies $\varphi(-,0)=f_0, \varphi(-,1)=f_1, \varphi(x_0,t)=y_0$ for all $t\in [0,1]$; 
equivalently, if some continuous map $\varphi^{\dagger}:X\to {\rm C}(I,Y)$ (where ${\rm C}(I,Y)$ carries the compact-open topology) with $\varphi^{\dagger}(-)(0)=f_0$ and 
$\varphi^{\dagger}(-)(1)=f_1$ maps $x_0$ to the constant map with value $y_0$. This latter presentation makes it easy to see that,
 when, in addition to $f_0\simeq f_1$ one is given $g_0 \simeq g_1: (Z,z_0)\to (Y,y_0)$, then $[f_0,g_0]\simeq [f_1,g_1]:
 (X,x_0)\vee (Z,z_0)\to (Y,y_0)$. Therefore, {\em ${\bf hTop_{\bullet}}$ has binary coproducts that are formed as in ${\bf Top_{\bullet}}$, and the same claim holds for arbitrary coproducts.}
 
 Consequently, when we consider $\sf S^1$ and $\gamma$ and $\tau$ as in (4), for every pointed space $X=(X,x_0)$ we obtain the {\em fundamental group}
 
 $$\pi_1(X)=\Omega X/\!\simeq \;= {\bf hTop}_{\bullet}({\sf S}^1,X)$$
 of $X$ with the usual concatenation operation
 
 $$\gamma_X: \pi_1(X)\times \pi_1(X)\to \pi_1(X), \; (a,b)\mapsto a*b: ({\sf S}^1\to^{\gamma} {\sf S}^1\vee {\sf S}^1 \to^{[a,b]} X).$$
 Consequently, with $\CT$ the theory of groups, ${\rm Aff}^*_{{\sf S}^1}(\CT,{\bf hTop}_{\bullet})$ has as objects pointed topological spaces $X$ that come with a subgroup $A$ of their fundamental group $\pi_1(X)$, and morphisms are
 homotopy classes of morphisms  in ${\bf Top}_{\bullet}$ that preserve the given subgroups.
\end{exmp}

The category ${\rm Aff}^*_S(\CT,\CX)$ inherits essential properties from $\CX$, because of the following easily proved and known (see \cite{Diers, GiuliTA}), but important, fact:

\begin{prop}\label{topological}
The forgetful functor $U:{\rm Aff}^*_S(\CT,\CX)\to \CX$ is topological (in the sense of {\rm \cite{AHS}}).
\end{prop}

\begin{proof}
Given any-size family $(Y_i,B_i)$ of dually affine $\CT$-spaces over $S$ and $\CX$-morphisms $f_i:X\to Y_i \;(i\in I)$, the $U$-initial structure on $X$ is 

$$A=\{a\in \CX(S,X) \mid \forall i\in I \;(f_i \cdot a \in B_i)\}=\bigcap_{i\in I}\CX(S,f_i)^{-1}(B_i).$$
Indeed, as an intersection of $\CT$-subalgebras $A$ is a $\CT$-subalgebra, and for any dually affine $\CT$-space
$(Z,C)$ and $h:Z\to X$ in $\CX$ one has

$$\{h\}\cdot C\subseteq A \;\Leftrightarrow \;\forall i\in I\;(\{f_i\cdot h\}\cdot C \subseteq B_i),$$
which confirms the $U$-initiality of the structure $A$.
\end{proof}

\begin{rem}\label{final}
For a family $(X_i,A_i)\in {\rm Aff}^*_S(\CT,\CX)$ and $f_i:X_i\to Y$ in $\CX\; (i\in I)$, the $U$-final stucture $B$ on $Y$ is the $\CT$-subalgebra of $\CX(S,X)$ generated by 
$\bigcup_{i\in I}\{f_i\}\!\cdot \! A_i.$ Note that each $\{f_i\}\!\cdot \! A_i=\CX(S,f_i)(A_i)$ is already a $\CT$-subalgebra, so that in the case of a singleton family no generation process is needed.
\end{rem}

\begin{cor}\label{limits}
Facilitated by $U$-initial and $U$-final liftings of, respectively, limit cones and colimit cocones in $\CX$, any type of limits or colimits existing in $\CX$ exists also in ${\rm Aff}^*_S(\CT,\CX)$ and is preserved by $U$. In fact, $U$ has a left- and a right-adjoint, which provide an object $X\in \CX$ with the discrete and the indiscrete structure, respectively, given by the least and the largest $\CT$-subalgebra of $\CX(S,X)$ (generated by $\emptyset$ and being $\CX(S,X)$ itself), respectively.
 \end{cor}
 
 The given distinguished object $S$ in $\CX$ becomes an object of ${\rm Aff}^*_S(\CT,\CX)$ when provided with the $\CT$-subalgebra $<\!1_S\!>$ generated by $\{1_S\}\subseteq \CX(S,S)$; we write
 
 $$S_1:=(S,<\!1_S\!>).$$
 An indication of the significance of the role of $S_1$ starts with the following easy observation.
 
 \begin{lem}\label{S as object}
 For any $(X,A)\in {\rm Aff}^*_S(\CT,\CX)$ and $a\in \CX(S,X)$ one has
 
 $$a\in A\;\Leftrightarrow \; a:S_1\to (X,A) \; {\rm in} \; {\rm Aff}^*_S(\CT,\CX).$$
   \end{lem}
\begin{proof}
Trivially, for $a\in \CX(S,X)$, the $\CT$-homomorphism $\CX(S,a):\CX(S,S)\to \CX(S,X)$ maps $\{1_S\}$ into $A$ if, and only if, $a\in A$, and then it must map even $<\!1_S\!>$ into $A$.
\end{proof}

\begin{cor}\label{rep}
The covariant hom-functor represented by $S_1\in {\rm Aff}^*_S(\CT,\CX)$ factors through ${\rm Alg}(\CT)$ and, with this codomain, is isomorphic to $\Gamma$. Consequently, one has the diagram
$$\bfig
\Atriangle/->`->`->/[{{\rm Aff}^*_S(\CT,\CX)}`\CX`{{\rm Alg}(\CT)};U`{\Gamma \cong {\rm Aff}^*_S(\CT,\CX)(S_1,-)}`{\CX(S,-)}]
\place(500,180)[\Longleftarrow] \place(500,240)[\iota]
\efig$$
depicting the natural transformation
$$ \iota:\Gamma\to \CX(S,U-),\;\iota_{(X,A)}:A\hookrightarrow \CX(S,X).$$
\end{cor}

The following result has been proved in \cite{Diers} for $\CX= {\bf Set}^{\rm op}$ (in the dual setting) but may be obtained quite generally.

\begin{thm}\label{left adjoint}
The structure functor $\Gamma: {\rm Aff}^*_S(\CT,\CX)\to {\rm Alg}(\CT)$ has a left adjoint, provided that $\CX$, besides all copowers of $S$, has also coequalizers.
\end{thm}

\begin{proof}
With $V:{\rm Alg}(\CT)\to {\bf Set}$ denoting the forgetful functor, $V\Gamma$ is representable by Corollary \ref{rep},
and by Corollary \ref{limits}, ${\rm Aff}^*_S(\CT,\CX)$ has all copowers of $S_1$, which guarantees the existence of a left adjoint to $V\Gamma$. Since ${\rm Aff}^*_S(\CT,\CX)$ has also coequalizers and $V$ is monadic, the (generalized version of) Dubuc's Adjoint Triangle Theorem (as given in Korollar (7) of \cite{Tholen triangle} and Exercise II.3.K(2) of \cite{monoidal topology}) assures us of the existence of a left adjoint to $\Gamma$.
\end{proof}

\begin{rem}\label{left adjoint explicit}
(1) For $\CX={\bf Set}^{\rm op}$ (see Remark \ref{Diers}), Diers \cite{Diers} gives an easy explicit description of the left adjoint to $\Gamma$. It assigns to a $\CT$-algebra $D$ the set $X(D)={\rm Alg}(\CT)(D,S)$, equipped with the 
$\CT$-subalgebra $A(D)=\{ \epsilon_D(d) \mid d\in D\} \le S^{X(D)}$, where $\epsilon_D(d): X(D)\to S$ is the evaluation map at $d$.

(2) In the general situation the proof of the Theorem gives the following recipe of how to construct a $\Gamma$-universal arrow for a $\CT$-algebra $D$, {\em i.e.}, for a set $D$ quipped with operations $\tilde{\omega}:D^{n_{\omega}}\to D$ for every given cooperation $\omega$ of $S$ in $\CX$. Consider the $\CT$-algebra that gives the structure of the copower $D\cdot S_1$ in 
${\rm Aff}^*_S(\CT,\CX)$, {\em i.e.}, the $\CT$-subalgebra $J$ of $\CX(S,D\cdot S)$ generated by the set of coproduct injections $j_d:D\to D\cdot S$ ($d\in D$).
One must now ``make" the map $D\to J, \; d\mapsto j_d,$ a $\CT$-homomorphism, by forming the joint coequalizer $q: D\cdot S\to Q$
of all pairs $(j_{\tilde{\omega}((d_i)_{i<n_{\omega}})},[j_{d_i}]_{i<n_{\omega}}\cdot \omega)$
of $\CX$-morphisms as depicted in the (generally non-commutative!) diagram

$$\bfig
\Atriangle/<-`->`->/[{n_{\omega}\cdot S}`S`{D\cdot S,};{\omega}`{[j_{d_i}]_{i<n_{\omega}}}`{j_{\tilde{\omega}((d_i)_{i<n_{\omega}})}}]
\efig$$
one pair for every given $\omega$ and every family $d_i\in D\;(i<n_{\omega})$. Then $D\to \Gamma(Q,\{q\}\cdot J),\; d\mapsto q\cdot j_d,$ is the desired $\Gamma$-universal arrow.
\end{rem}

There is an important special case when no coequalizers are needed, that is, when the $\CT$-algebra $D$ is free, so that
$D\cong Fn$ for some set $n$ and $F \dashv V: {\rm Alg}(\CT)\to{\bf Set}$. Indeed, since $V\Gamma$ is represented by $S_1$, the $n$-th copower of $S_1$ in ${\rm Aff}^*_S(\CT,\CX)$ is the only candidate for the $\Gamma$-universal object over $Fn$:

\begin{cor}\label{Gamma universal for free}
If the $n$-th copower of $S$ exists in $\CX$, then there is a $\Gamma$-universal arrow for the free $\CT$-algebra $Fn$ over the set $n$, given by the $\CT$-homomorphism

$$\kappa_n:Fn\to J_n=\Gamma(n\cdot S_1,J_n),\,i\mapsto j_i\,(i \in n),$$
where $J_n$ is the $\CT$-subalgebra of $\CX(S,n\cdot S)$ generated by the coproduct injections $j_i:S\to n\cdot S,\,i\in n$.
\end{cor}

\begin{proof}
Given $(Y,B)$ in ${\rm Aff}^*_S(\CT,\CX)$, a $\CT$-homomorphism $\varphi:Fn\to B$ is determined by a family of morphisms $b_i:S\to Y\,(i \in n)$ in $B$. By the representability of $\Gamma$, $f=[b_i]_{i\in n}:n\cdot S\to Y$ in $\CX$ gives the only morphism 
$n\cdot S_1\to(Y,B)$ in ${\rm Aff}^*_S(\CT,\CX)$
with $\Gamma f\cdot \kappa_n=\varphi.$
\end{proof}

\begin{exmp}\label{examples second}
(1) In Example \ref{examples first}(1), in the absence of any given cooperations, $q$ of Remark \ref{left adjoint explicit} may be taken as an identity map. Consequently, the left adjoint of

$$\Gamma:{\rm Aff}^*_1(\emptyset, {\bf Set})={\rm Sub}({\bf Set})\to {\rm Alg}(\emptyset)={\bf Set}$$
is trivial ($D\mapsto (D,D)$) -- a fact that, of course, is also easily seen directly. It embeds ${\bf Set}$ into ${\rm Sub}({\bf Set})$ as a full coreflective subcategory.

(2) The left adjoint of

$$\Gamma:{\rm Aff}^*_R(\emptyset, {\bf Mod}_R)\to {\rm Alg}(\emptyset)={\bf Set}$$
pertaining to Example \ref{examples first}(2) assigns to a set $D$ the free $R$-module $D\cdot R$ provided with its
standard basis as the distinguished subset and provides again a full coreflective embedding.

(3) In Example \ref{examples first}(3), $q:D\cdot R \to D$ of Remark \ref{left adjoint explicit} is simply the counit at 
$D\in {\bf Mod}_R$ of the adjunction $F \dashv V: {\bf Mod}_R\to {\bf Set}$. As a consequence, the left adjoint of

$$\Gamma:{\rm Aff}^*_R(\CT, {\bf Mod}_R)={\rm Sub}({\bf Mod}_R)\to {\rm Alg}(\CT)={\bf Mod}_R$$
maps as in (1), so that $D\mapsto (D,D)$, thus again providing the obvious full coreflective embedding.

(4) For Example \ref{examples first}(4), the quotient map $q$ of Remark \ref{left adjoint explicit} is much harder to compute than in the previous three situations. However, $q$ remains easily describable when the given $D\in {\rm Alg}(\CT)$, {\em i.e.}, the non-empty set $D$ with a binary operation and a self-inverse unary operation, is the initial  or the terminal object
in ${\rm Alg}(\CT)$, denoted here by $D_0$ and 1, respectively. Indeed, the left adjoint of $\Gamma$ must assign to $D_0$ the initial object in ${\rm Aff}^*_{{\bf S}^1}(\CT,{\bf Top}_{\bullet})$,
given by $(\bullet,D_0)$, {\em i.e.}, by the zero object of ${\bf Top}_{\bullet}$ provided with the initial $\CT$-algebra.

For $D=1$ terminal, $q:{\sf S}^1\to Q$ is the joint coequalizer of the pairs$(1_{{\sf S}^1},\nu)$, $(1_{{\sf S}^1},\tau)$ and $(1_{{\sf S}^1}, \delta)$, where $\nu$ is the constant map ${\sf S}^1 \to {\sf S}^1$ and $\delta$, when we present ${\sf S}^1$ as ${\mathbb R}/{\mathbb Z}$, 
is described by $(t+{\mathbb Z}\mapsto 2t+{\mathbb Z})$. While the $\CT$-algebra $J$ generated by $1_{{\sf S}^1}$ in ${\bf Top}_{\bullet}({\sf S}^1,{\sf S}^1)$ is the free $\CT$-algebra on one generator, $Q$ is 
terminal in ${\bf Top}_{\bullet}$ and, hence,
$\{q\}\cdot J\subseteq {\bf Top}_{\bullet}({\sf S}^1,Q)$ 
is the terminal $\CT$-algebra, making also the unit of the adjunction at 1 trivial: $1\to \Gamma (\bullet,1)=1$, as one should have expected. 

(5) For Example \ref{examples first}(5), because of the missing coequalizers in ${\bf hTop}_{\bullet}$, Theorem \ref{left adjoint} does {\em not} assure us of the right adjointness of the group-valued functor

$$\Gamma: {\rm Aff}^*_{ {\sf S}^1}(\CT, {\bf hTop}_{\bullet}) \to {\rm Alg}(\CT)={\bf Grp}.$$
However, we are still able to apply
 Corollary \ref{Gamma universal for free} and produce $\Gamma$-universal arrows for {\em free} groups on $n$  generators. In fact, since 
 
 $$\pi_1(n\cdot {\sf S}^1)\cong Fn$$
 is freely generated by the coproduct injections of $n\cdot {\sf S}^1$, 
 the natural isomorphism 
 
 $$Fn\to \Gamma (n\cdot {\sf S}^1,\pi_1(n\cdot {\sf S}^1))$$
  serves as the $\Gamma$-universal
 arrow for $Fn$.
\end{exmp}

\section{The Zariski dual closure operator, separation and completeness}
Regular epimorphisms $p:(X,A)\to (P,C)$ in the topological category ${\rm Aff}^*_S(\CT,\CX)$ over $\CX$ are described as regular epimorphisms $p:X\to P$ in $\CX$ with $C=\{p\}\cdot A$. The {\em Zariski dual closure} of $p$
--if it exists-- is the regular epimorphism $\zeta_{(X,A)} p=\zeta p$ with domain $(X,A)$ characterized by the following two properties:

1. $\forall a,b\in A\,(p\cdot a = p\cdot b \Rightarrow \zeta p\cdot a= \zeta p\cdot b)$;

2. every $\CX$-morphism $f$ with domain $X$ satisfying  $(\forall a,b\in A\,(p\cdot a = p\cdot b \Rightarrow f\cdot a= f\cdot b))$ factors through $\zeta p$.

The following proposition gives conditions for the existence of $\zeta$ and establishes it as an idempotent dual closure operator
(in the sense of \cite{dual clop}) for regular epimorphisms of ${\rm Aff}^*_S(\CT,\CX)$. Of course, the proposition follows from the known dual facts, but we find it helpful in the examples to spell these out explicitly in the current setting. For regular epimorphisms $p,p'$ in $\CX$ with the same domain we write $p\leq p'$ if $p'$ factors through $p$ (so that $p'=h\cdot p$ for some $h$). 
For $f:X\to Y$ and a regular epimorphism $q$ with domain $Y$,  $f^-(q)$ with domain $X$ is defined as --existence granted-- the regular-epi part in the (regular epi, mono)-factorization of $q\cdot f$.

\begin{prop}\label{zeta closure}
Let $\CX$ have coequalizers and (regular epi, mono)-factorizations, as well as copowers of $S$ or co-intersections (= wide pushouts) of small families of regular epimorphisms with common domain. Then every regular epimorphism in ${\rm Aff}^*_S(\CT,\CX)$ 
has a Zariski dual closure, subject to the following rules: 
\begin{itemize}
\item[{\rm (1)}] $\zeta p\le p$;
\item[{\rm (2)}] $p\le p'\Rightarrow \zeta p\le \zeta p'$;
\item[{\rm (3)}] $\zeta p\le \zeta \zeta p$;
\item[{\rm (4)}] $\zeta_{(X,A)}(f^-(q))\le f^-(\zeta_{(Y,B)}q)$,
\end{itemize}
for all morphisms $f:(X,A)\to (Y,B)$ and regular epimorphisms $p,p'$ with domain $(X,A)$ and $q$ with domain $(Y,B)$.
\end{prop}

\begin{proof}
With the notation
 
$${\rm ker}_A(p):=\{(a,b)\in A\times A \mid p\cdot a=p\cdot b\},$$ 
the underlying $\CX$-morphism of $\zeta p$ may either be constructed as the
coequalizer of the induced morphisms $\alpha, \beta: ({\rm ker}_A(p))\cdot S\to X$ with $\alpha = [a]_{(a,b)\in {\rm ker}_A(p)}, \beta = [b]_{(a,b)\in {\rm ker}_A(p)}$, or as the co-intersection of the family $(e_{a,b})_{(a,b)\in {\rm ker}_A(p)}$, with $e_{a,b}$ the coequalizer of $a, b:S\to X$. 

Showing that $\zeta p$ satisfies the characteristic properties 1 and 2 is a routine diagram chase, and so are the verifications of 
(1) and (2). Rule (3) follows from the characteristic property 2 once one has noticed that ${\rm ker}_A(p)={\rm ker}_A(\zeta p)$. Similarly, for (4) one must show ${\rm ker}_A(f^-(q))\subseteq {\rm ker}_A(f^-(\zeta q))$. Indeed, if 
$f^-(q)\cdot a= f^-(q)\cdot b$, then $(f\cdot a,f\cdot b)\in {\rm ker}_B(q)={\rm ker}_B(\zeta q)$, which implies 
$(a,b)\in {\rm ker}_A(f^-(\zeta q))$.
\end{proof} 

\begin{defn}(\cite{dual clop}) For a regular epimorphism $p$ in ${\rm Aff}^*_S(\CT, \CX)$ with domain $(X,A)$, let 
$\theta p\cdot \zeta p =p$ be the factorization of $p$ through its (existing) Zariski dual closure. One then calls $p$
$\zeta$-{\em closed} if $\theta p$ is an isomorphism, and $p$ is $\zeta$-{\em sparse} if $\zeta p$ is an isomorphism.
\end{defn}

\begin{rem}\label{rem zeta-closed}
Already being a regular epimorphism when $p$ is one, $\theta p$ or $\zeta p$ will be an isomorphism as soon as it is a monomorphism (in ${\rm Aff}^*_S(\CT, \CX)$ or, equivalently, in $\CX$). Also the following statements follow immediately from the definitions or the preceding statements:

(1) $p$ is $\zeta$-closed if, and only if, in the notation of the proof of Proposition \ref{zeta closure}, every $f:X \to Y$ in $\CX$ with 
${\rm ker}_A(p)\subseteq {\ker}_A(f)$ factors through $p$. Note that
every epimorphism $p$ satisfying this characteristic property  of $\zeta$-closedness must automatically be regular (in the sense of \cite{Kelly}).

(2) $p$ is $\zeta$-sparse if, and only if, ${\rm ker}_A(p) \subseteq \Delta_A$, with $\Delta_A$ the identity relation on the set $A$; if $p$ is also $\zeta$-closed, it must be an isomorphism.

(3) $\zeta p$ is $\zeta$-closed, and $\theta p$ is $\zeta$-sparse, for every $p$.
\end{rem}

\begin{exmp}\label{examples third} We refer to Example \ref{examples first}.

(1) For $p:(X,A)\to (Y,B)$ in ${\rm Aff}_1(\emptyset,{\bf Set})$ with $p$ surjective, $\zeta p$ is given by the map
$X\to p(A)+(X\setminus A)$ that maps elements in $A$ like $p$ does but maps elements in $X\setminus A$ identically. Consequently, $p$ is $\zeta$-closed precisely when $p\,|_{X\!\setminus \!A}$ is injective, and $\zeta$-sparse when $p|_A$ is injective.

(2) Keeping the same notation, let $p$ now be in ${\rm Aff}_R(\emptyset,{\bf Mod}_R)$. Then $\zeta p$ is described by the projection $X\to X/{\hat{A}}$, with $\hat{A}$ the submodule generated by $\{a-b \mid a,b\in A, p(a)=p(b)\}$.
In this description $p$ is $\zeta$-closed precisely when ${\rm ker}p \subseteq \hat{A}$, and $\zeta$-sparse when
$\hat{A}=0$. 

(3) For $p$ in ${\rm Aff}_R(\CT, {\bf Mod}_R)$ with $\CT$ the theory of $R$-modules, $\zeta p$ is given by the projection $X\to X/({\rm ker}p\cap A)$. Now $p$ is $\zeta$-closed ($\zeta$-sparse) if, and only if, ${\rm ker}p \subseteq A$
(${\rm ker}p \cap A=0$, respectively).

\end{exmp}

Next we will apply the $\zeta$-closure to the counit of the representable functor 
$V\Gamma \cong {\rm Aff}^*_S(\CT,\CX)$ (see Corollary \ref{rep}) at a dually $\CT$-affine space $(X,A)$, and for that it will be useful to examine first the role of $S_1=(S,<1_S>)$ beyond Lemma \ref{S as object}.

\begin{prop}\label{finally dense}
For every object $(X,A)$, the family of morphisms $a:S_1\to (X,A)\, (a\in A)$ is final with respect to the topological functor 
$U:{\rm Aff}^*_S(\CT,\CX) \to \CX$. Consequently, $S_1$ is $U$-finally dense in ${\rm Aff}^*_S(\CT,\CX)$.
\end{prop}

\begin{proof} 
For all $a\in A$ one has $a\in \{a\}\cdot <1_S>\le A$. The $\CT$-algebra $A$ is therefore generated by $\bigcup_{a\in A}\{a\}\cdot <1_S>$, which is the $U$-final structure.
\end{proof}

\begin{cor}\label{counit final}
Existence of the needed copowers granted, for every object $(X,A)$ the morphism 

$$\varepsilon_{(X,A)}=\varepsilon: A\cdot S_1\to (X,A)\, {\rm with}\,\, \varepsilon \cdot j_a=a\,\,(a\in A)$$
($j_a:S\to A\cdot S$ denoting a coproduct injection) is $U$-final.
\end{cor}

\begin{rem}
{\em The family $a:S_1\to (X,A)\,(a\in A)$ is $U$-initial for every object $(X,A)$ if, and only if, $<1_S>=\CX(S,S)$, {\rm i.e.}, if $S_1$ carries the largest possible $\CT$-algebra as its structure}. Indeed, as the $U$-initial structure is given by
$\bigcap_{a\in A}\CX(S,a)^{-1}(A)\ge <1_S>$, 
the condition $<1_S>=\CX(S,S)$ is certainly sufficient for the $U$-initiality of $a:S_1\to (X,A)\,(a\in A)$. Considering $(X,A)=(S,\CX(S,S))$ one sees that it is also necessary. 
\end{rem}

\begin{defn}\label{sep and compl}
Let $S$ have all copowers in $\CX$. A dually $\CT$-affine algebra $(X,A)$ modelled by $S$ is called
\begin{itemize}
\item {\em separating} if any two morphisms $g,h:(X,A)\to (Y,B)$ must be equal whenever $g\cdot a=h\cdot a$ for all
 $a\in A$; equivalently, if $\varepsilon_{(X,A)}$ is epic in ${\rm Aff}^*_S(\CT,\CX)$ or, equivalently,  in $\CX$;
 \item {\em regularly separating} if $\varepsilon_{(X,A)}$ is a regular epimorphism in ${\rm Aff}^*_S(\CT,\CX)$ or, equivalently,  in $\CX$; 
 \item {\em $\zeta$-complete} if $\varepsilon_{(X,A)}$ is a $\zeta$-closed regular epimorphism; that is (see Remark \ref{rem zeta-closed}): if $(X,A)$ is separating, and if every $f:A\cdot S\to Y$ in $\CX$ with 
 
 $$\forall s,t \in J_A:=<j_a \mid a\in A>\le \CX(S,A\cdot S)\,\,(\varepsilon \cdot s=\varepsilon \cdot t\Rightarrow f\cdot s=f\cdot t)$$
 
factors through $\varepsilon =\varepsilon_{(X,A)}$. (In what follows, we will keep the notation $J_A$ for the $\CT$-subalgebra generated by the coproduct injections $j_a (a\in A)$).
 
 \end{itemize}
\end{defn}

\begin{rem}\label{trivial}
(1) The full subcategory of separating objects in ${\rm Aff}^*_S(\CT,\CX)$ is easily seen to be closed under
epi-sinks, in particular closed under colimits, and therefore coreflective in ${\rm Aff}^*_S(\CT,\CX)$ under mild hypotheses on $\CX$. Indeed, for a jointly epic family $f_i:(X_i,A_i)\to (Y,B)$, when the family 
$a:S_1\to (X,A)\,(a\in A)$,
is epic, so is $f_i\cdot a \,(a\in A, i\in I)$, which is subfamily of $b:S_1\to (Y,B)\,(b\in B)$.

Similarly, regularly separating objects can be seen to be closed under regular epi-sinks, under mild hypotheses on 
$\CX$.

(2) While there is no comparable easy stability property for $\zeta$-completeness as there is for separating objects (but see the characterization via projectivity in Proposition \ref{projective}(4) below!), we should point out that the notion of $\zeta$-completeness becomes rather simple when $\CT=\emptyset$ (or the initial theory). Indeed, in this case one has $J_A = \{j_a \mid a\in A\}$ and therefore trivially
${\rm ker}_{J_A}(\varepsilon_{(X,A)})\subseteq \Delta_{J_A}$ for all objects $(X,A)$; consequently, for $(X,A)$  regularly separating, the morphism $\varepsilon_{(X,A)}$ is in fact $\zeta$-sparse, hence an isomorphism if requested to be also 
$\zeta$-closed. Briefly: {\em $(X,A)$ is $\zeta$-complete if, and only if, $\varepsilon_{(X,A)}$ is an isomorphism in 
${\rm Aff}^*_S(\emptyset,\CX)$  or, equivalently, in $\CX$}.
\end{rem}

Here is a characterization of (regularly) separating and of $\zeta$-complete objects that utilizes the role of $S_1$ in ${\rm Aff}^*_S(\CT,\CX)$ that is known in the dual situation when $\CX={\bf Set}^{\rm op}$ -- see, for example,  \cite{Giuli Hofmann} --, which is why we can keep its proof rather short.

\begin{prop}\label{projective}
{\rm (1)} $S_1$ is projective in ${\rm Aff}^*_S(\CT,\CX)$ with respect to the class of $U$-final morphisms and, in particular, the class of regular epimorphisms, and so are all of its copowers.

{\rm (2}) An object $(X,A)$ is separating if, and only if, every $U$-final morphism $h$ in  ${\rm Aff}^*_S(\CT,\CX)$
with codomain $(X,A)$ is an epimorphism.

{\rm (3)} If $\CX$ has (regular epi, mono)-factorizations, then $(X,A)$ is regular separating precisely when every $U$-final morphism $h$ in  ${\rm Aff}^*_S(\CT,\CX)$
with codomain $(X,A)$ is a regular epimorphism.

{\rm (4)} Existence of the needed $\zeta$-closures granted, a separating object $(X,A)$ is $\zeta$-complete if,
and only if, it is projective with respect to $\zeta$-sparse regular eopimorphisms.
\end{prop}

\begin{proof}
(1) Given $f:(X,A)\to (Y,B)$ $U$-final and $g:S_1\to (Y,B)$, one has $g\cdot 1_S\in B=\{f\}\cdot A$, so that 
$g$
factors as $g=f\cdot a$ with $a\in A$, {\em i.e.}, with $a$ a morphism $S_1\to (X,A)$ by Lemma \ref{S as object}.
Furthermore, projectivity is a property stable under taking coproducts.

(2), (3) The condition is necessary since, given $h:(Z,C)\to (X,A)$ $U$-final, the projectivity assertion of (1) makes $\varepsilon: A\cdot S \to X$ factor through $h$. Consequently, $h$ must be epic when $\varepsilon$ is, and the same conclusion can be drawn in the ``regular case", provided that the class of regular epimorphisms in $\CX$ is right cancellable -- which is certainly guaranteed in the presence of (regular epi, mono)-factorizations. Conversely, one simply exploits the given property for $h=\varepsilon_{(X,A)}$, which is $U$-final by Corollary \ref{counit final}.

(4) To show the necessity of the condition, consider morphisms $p:(Y,B)\to (Z,C), f:(X,A)\to (Z,C)$ with $p$ $\zeta$-sparse. $U$-finality of $p$ makes all $f\cdot a, \,a\in A$, factor through $p$, whence also $f\cdot \varepsilon_{(X,A)}$ 
factors through $p$. With $p$ being $\zeta$-sparse and $\varepsilon_{(X,A)}$
$\zeta$-closed, this implies that $f$ factors through $p$, by the standard ``diagonalization property".

Conversely, assuming $(X,A)$ to be projective as indicated, first observe that any $\zeta$-sparse regular epimorphism $q:(Q,D)\to (X,A)$ with $(Q,D)$ separating must be an isomorphism. This easily shown fact may
then be applied to $q=\theta \varepsilon_{(X,A)}$ whose domain, as a quotient of $A\cdot S$, is indeed separating, as
we confirm in the proof of the theorem that follows.
\end{proof}

\begin{thm}\label{S1 complete}
Let $\CX$ have all copowers of the distinguished object $S$. Then $S_1$ is a regular-projective (regular) generator of the full subcategory of (regularly) separating objects in ${\rm Aff}^*_S(\CT,\CX)$. More importantly, $S_1$ and all of its copowers are $\zeta$-complete.
\end{thm}

\begin{proof}
For a set $n$ let $(X,A)=n\cdot S_1$ be the $n$-th copower of $S_1$ in ${\rm Aff}^*_S(\CT,\CX)$, hence $X=n\cdot S$ with
coproduct injections $h_i:S\to n\cdot S, \,i\in n$, which generate the $\CT$-subalgebra $A=<h_i\mid i\in n>$ of 
$\CX(S,X)$. The morphism $\varepsilon =\varepsilon_{(X,A)}: A\cdot S\to X$ with $\varepsilon \cdot j_a, \, j_a$ the coproduct injections of $A\cdot S \,(a\in A)$, is certainly a regular epimorphism since it actually splits in $\CX$. Indeed, the splitting is provided by the morphism $d: n\cdot S\to A\cdot S$ with $d\cdot h_i = j_{h_i},\, i\in n$. The
first claim of the Theorem now follows from the case $n=1$ in conjunction with Proposition \ref{projective}.

Next we note that $d\cdot a$ lies in $J_A=<j_a \mid a\in A>$ for all $a\in A$ since $d\cdot h_i=j_{h_i}$ does, for all $i\in n$ (so that $d: (X,A)\to A\cdot S_1$ is actually an
${\rm Aff}^*_S(\CT,\CX)$-morphism). 
To show that $\varepsilon$ is $\zeta$-closed we consider any $f:A\cdot S \to  Y$ in $\CX$ with 
${\rm ker}_{J_A}(\varepsilon)\subseteq {\rm ker}_{J_A}(f)$. From

$$\varepsilon\cdot (d \cdot \varepsilon \cdot j_a)= \varepsilon\cdot d\cdot a=a=\varepsilon\cdot j_a$$
we then obtain 
$(f\cdot d \cdot \varepsilon) \cdot j_a=f\cdot j_a$
for all $a\in A$ and, hence, $(f\cdot d)\cdot \varepsilon=f$, so that $f$ factors through $\varepsilon$. Consequently, $\varepsilon$ is $\zeta$-closed, and $n\cdot S_1$ is $\zeta$-complete.
\end{proof}

\begin{exmp}\label{examples fourth}
We refer to Example \ref{examples first}.

(1) $(X,A)$ in ${\rm Sub}({\bf Set})$ is (regularly) separating if, and only if $A=X$, and is then already $\zeta$-complete.

(2) $(X,A)$ in ${\rm Aff}^*_R(\emptyset, {\bf Mod}_R)$ is (regularly) separating if, and only if, $A$ generates $X$ as an $R$-module, and $(X,A)$ is $\zeta$-complete if, and only if $A$ is a basis of the $R$-module $X$. Hence, to be 
the under lying $R$-module of a $\zeta$-complete object it is necessary and sufficient to be free.

(3) For $(X,A)$ in ${\rm Sub}({\bf Mod}_R)$, in the notation of Definition \ref{sep and compl} one has 
$J_A=A\cdot R$. It is therefore easy to see that (regularly) separating as well as $\zeta$-complete objects are characterized as in (1): $A=X$.

\end{exmp}

For $\CT$ the theory of groups and $\CX={\bf hTop}_{\bullet}$, we already saw in Example \ref{examples second}(5)
that the $n$-th copower $n\cdot {\sf S}^1$ of the 1-sphere ${\sf S}^1$ (with $n$ any set) is $\Gamma$-universal when provided with its fundamental group. But $\pi_1(n\cdot {\sf S}^1)$ is freely generated by the coproduct
injections of $n\cdot {\sf S}^1$. Hence, with Theorem \ref{S1 complete} we conclude:

\begin{cor}\label{final cor}
$(n\cdot {\sf S}^1, \pi_1(n\cdot {\sf S}^1))$ is $\zeta$-complete in ${\rm Aff}^*_{{\sf S}^1}(\CT,{\bf hTop}_{\bullet})$,
for all sets $n$.
\end{cor}

\section*{References} 

%% If you have bibdatabase file and want bibtex to generate the
%% bibitems, please use
%%
%\bibliographystyle{elsarticle-num}
%\bibliographystyle{abbrv}
%\bibliography{lili}

\begin{thebibliography}{10}

\bibitem{AHS}
J.~Ad{\'a}mek, H.~Herrlich, and G.~E. Strecker.
\newblock {\em Abstract and Concrete Categories: The Joy of Cats}.
\newblock Wiley, New York, 1990.

\bibitem{Arkowitz}
M.~Arkowitz.
\newblock {\em Introduction to Homotopy Theory}.
\newblock Springer, New York, 2011.

\bibitem{BGH}
G.C.L.~Br\"ummer, E.~Giuli, H.~Herrlich.
\newblock Epireflections which are completions.
\newblock {\em Cahiers Topologie G\'eom. Differentielle Cat\'egoriques} 33:71--93, 1992.

\bibitem{DGLC}
D.~Deses, E.~Giuli, E.~Lowen-Colebunders.
\newblock On complete objects in the category of ${\rm T}_0$-closure spaces.
\newblock {\em Appl. General Topology} 4:25--34, 2003.

\bibitem{Diers}
Y.~Diers.
\newblock Categories of algebraic sets.
\newblock {\em Appl. Categorical Structures}, 4(2--3):329--341, 1996.

\bibitem{DiersGeom}
Y.~Diers.
\newblock Affine algebraic sets relative to an algebraic theory.
\newblock {\em J. Geometry}, 65:54--76, 1999.

\bibitem{Dimov Tholen}
G.D.~Dimov and W.~Tholen.
\newblock A characterization of representable dualities. In:
\newblock {\em Conf. Categorical Topology and its Relations to Analysis, Algebra and Combinatorics, Prague 1988},
\newblock pp. 336--357, World Scientific, Singapore, 1989.

\bibitem{clop}
D.~Dikranjan and E.~Giuli.
\newblock Closure operators I.
\newblock {\em Topology Appl.} 27:129--143, 1987.

\bibitem{dual clop}
D.~Dikranjan and W.~Tholen.
\newblock Dual closure operators and their applications.
\newblock {\em J. Algebra}, 439:373--416, 2015.

\bibitem{Dugundij}
J.~Dugundij.
\newblock {\em Topology}.
\newblock Allyn and Bacon, Boston, 1966.

\bibitem{GiuliTA}
E.~Giuli.
\newblock Zariski closure, completeness and compactness.
\newblock {\em Topology Appl.} 153:3158--3168, 2006.

\bibitem{GiuliT0}
E.~Giuli.
\newblock On classes of ${\rm T}_0$-spaces admitting completions.
\newblock {\em Appl. General Topology} 4:143--155, 2003.

\bibitem{Giuli Hofmann}
E.~Giuli and D.~Hofmann.
\newblock Affine sets: The structure of complete objects and duality.
\newblock{\em Topology Appl.} 156:2129--2136, 2009.

\bibitem{monoidal topology}
D.~Hofmann, G.~J. Seal, and W.~Tholen, {\em editors}.
\newblock {\em Monoidal Topology: A Categorical Approach to Order, Metric, and  Topology}, volume 153 of {\em Encyclopedia of Mathematics and its Applications}.
\newblock Cambridge University Press, Cambridge, 2014.

%%\bibitem{Hu}
%%S.-T. Hu.
%%\newblock {\em Homotopy Theory}
%%\newblock Academic Press,New York, 1959.

\bibitem{Kelly}
G.M.~Kelly.
\newblock Monomorphisms, epimorphisms and pullbacks.
\newblock {\em J. Austral. Math. Soc.} 9:124--142, 1969.

\bibitem{Mac Lane}
S.~Mac Lane.
\newblock {\em Categories for the Working Mathematician} (2nd ed.).
\newblock Springer, New York, 1998.

\bibitem{Manes}
E.G.~Manes.
\newblock {\em Algebraic Theories}.
\newblock Springer, New York, 1976.

\bibitem{Porst Tholen}
H.-E.~Porst and W.~Tholen.
\newblock Concrete dualities. In:
\newblock H.~Herrlich and H.-E.~Porst (eds.),
\newblock {\em Category Theory at Work}, pp. 111--136, Heldermann, Berlin, 1991.

\bibitem{Strom}
J.~Strom.
\newblock {\em Modern Classical Homotopy Theory}.
\newblock American Mathematical Soc., Providence RI, 2011.

\bibitem{Tholen triangle}
W.~Tholen.
\newblock Adjungierte Dreiecke, Colimites and Kan-Erweiterungen.
\newblock {\em Math. Ann.} 217:121--129, 1975.

\bibitem{Willard}
S.~Willard.
\newblock {\em General Topology}.
\newblock Addison Wesley, Reading MA, 1970.

\end{thebibliography}

\end{document}